\documentclass[a4paper,11pt]{scrartcl}
\usepackage{pgfplots}
\usepackage[utf8]{inputenc}
\usepackage{lmodern}
\usepackage[english]{babel}
\usepackage[unicode, pdftex,
            pdfauthor={Wodson Mendson},
            pdftitle={Arithmetic aspects of the polynomial maps},
            pdfsubject={number theory, algebraic geometry},
            pdfkeywords={unimodular, primes, polynomial map},
            pdfproducer={Latex with hyperref},
            pdfcreator={Overleaf}]{hyperref}
\usepackage{amsmath}
\usepackage{amsfonts}
\usepackage{amssymb}
\usepackage{amsthm}
\usepackage{MnSymbol}
\usepackage{titling}
\usepackage{enumerate}

\newtheorem*{resumo}{Abstract}
\newtheorem*{unimodularconjecture}{Unimodular Conjecture}
\newtheorem*{unimodularII}{Unimodular Conjecture II}

\newtheorem{thm}{Theorem}
\newtheorem*{thmo}{Theorem}
\newtheorem*{thmA}{Theorem A}

\newtheorem*{cor}{Corollary}

\newtheorem*{Hlemma}{Hensel's Lemma}

\newtheorem*{jacobian}{Jacobian Conjecture}
\newtheorem*{jacobianii}{Conjecture}

\def\idd#1{\ensuremath{\langle{#1}\rangle}}

\newtheorem{lemma}{Lemma}
\newtheorem{prop}{Proposition}
\newtheorem{OBS}{Remark}
\newtheorem{ex}{Example}
\newtheorem{dfn}{Definition}
\newtheorem{pro}{Problem}

\DeclareMathOperator{\corpo}{k}

\DeclareMathOperator{\Spec}{\textbf{\mbox{Spec}}}
\DeclareMathOperator{\spm}{\textbf{Spm}}
\DeclareMathOperator{\ch}{\textbf{char}}

\DeclareMathOperator{\field}{k}

\author{Wodson Mendson}
\title{Some arithmetic aspects of polynomial maps}

\begin{document}

\maketitle

    \begin{resumo}  \normalfont     
        The Jacobian conjecture is a well-known open problem in affine algebraic geometry that asks if any polynomial endomorphism of the affine space $\mathbb{A}_{\mathbb{C}}^{n}$ ($n\geq2$) with jacobian $1$ is an automorphism. We present a survey about some results around this conjecture and we discuss an arithmetic aspect of this conjecture due to Essen-Lipton.  We investigate some cases of this arithmetic approach showing the close relationship between the Jacobian Conjecture and the problem of counting $\mathbb{F}_p$-points of an affine scheme.
        \end{resumo}
\setcounter{tocdepth}{1}
\tableofcontents

\section{Introduction}
Let $\field$ be an algebraically closed field of characteristic zero. The Jacobian Conjecture over $\field$ is a classical problem in affine algebraic geometry. It was first formulated by Ott-Heinrich Keller in $1939$ and asks whether a polynomial endomorphism $\psi$ of the $n$-dimensional affine space $\mathbb{A}_{\field}^{n} = \textbf{\mbox{Spec}}(\field[x_{1},\ldots,x_{n}])$ with
$$
\det((\frac{\partial\psi^{*}(x_{i})}{\partial x_{j}})_{1\leq i,j\leq n}) = 1
$$
is an automorphism, that is, there is a polynomial endomorphism $\gamma: \mathbb{A}_{\field}^{n} \longrightarrow \mathbb{A}_{\field}^{n}$ such that 
$$\gamma\circ \psi = id_{\mathbb{A}_{\field}^{n}} = \psi\circ\gamma.$$ 

 In \cite{poly1}, H. Bass, E. H. Connell, and D. Wright showed  that to prove the Jacobian Conjecture 
it is enough to prove it for endomorphism in the form $\psi =
X + H = (x_1+h_1,\ldots,x_n+h_n)$, where $h_i$ is 
homogeneous of degree $3$ for every $i$ with the Jacobian matrix of $H = (h_1,\ldots, h_n)$ nilpotent. A refinement, due to Essen-Bondt, ensures that it is sufficient to consider maps in the form $F = X + H$ with $H = (h_{1},\ldots,h_{n})$ homogeneous, $\deg(H) = 3$ with $H$ having jacobian matrix nilpotent \textbf{symmetric} (see \cite[Theorem 1.1]{de2005reduction}). 

An interesting arithmetical approach to the Jacobian Conjecture was explored on \cite{poly2}. In that paper, it is proved that the Jacobian Conjecture is related, in some sense, to a problem about counting $\mathbb{F}_p$-rational point of an affine scheme. To explain, let us consider the basic example: let $R$ be a local domain with maximal ideal $\mathbf{m}$ and residue field $\field$. Let $f_1,\ldots,f_n \in R[x_{1},\dots ,x_{n}]$ be homogeneous of degree one. Assume the Jacobian matrix associated with $f_1,\ldots,f_n$ is invertible. Denote this Jacobian matrix by $J$. Let $M \in \mathcal{M}_{n}(R)$ be the matrix such that $ JM = MJ = 1$. The
relation $M\cdot J = id$ implies that there exist $u_{1},\dots ,u_{n} \in R$ 
such that $f_{1}(u_{1},\dots ,u_{n}) = 1$. In particular,   since $\overline{f}(\overline{u}_1,\ldots, \overline{u}_n) \neq 0$ the induced polynomial map 
$$
f = (f_1,\ldots,f_n)\colon \mathbb{A}_{R}^{n}\longrightarrow \mathbb{A}_{R}^{n}
$$
induces a polynomial map:$\overline{f} = f\otimes R/\mathfrak{m}\colon \mathbb{A}_{\field}^{n} \longrightarrow \mathbb{A}_{\field}^{n}$ that is non-zero.
For degrees bigger than $1$ we have the following conjecture formulated by Essen-Lipton (see \cite{poly2}).

\begin{unimodularconjecture} Let $R$ be a local domain of characteristic zero with maximal ideal $\mathbf{m}$ and residue field $\field$. Then, for every $n\in \mathbb
{Z}_{>2}$ and polynomial map with jacobian invertible
$$
    f = (f_1,\ldots,f_n): \mathbb{A}_{R}^{n}\longrightarrow \mathbb{A}_{R}^{n}
$$
the induced polynomial map obtained by the reduction modulo $\mathfrak{m}$
$$
    f= (\overline{f_1},\ldots,\overline{f_n})\colon \mathbb{A}_{\field}^{n} \longrightarrow \mathbb{A}_{\field}^{n}
$$
is non-zero. 
\end{unimodularconjecture}

A simple fact is that if $\field$ is an infinite field then $R$ satisfies the Unimodular Conjeture (see Proposition \ref{casoinfinito}). When $\field$ is finite we get the following reformulation regarding rational points.

\begin{unimodularII} Let $R$ be local domain of characteristic zero with maximal ideal $\mathbf{m}$ and finite residue field $\field$. Let $f_1,\ldots,f_n \in R[x_1,\ldots,x_n]$ be polynomials and consider the affine scheme $X = \Spec(A)$ where $A = R[x_1,\ldots,x_n]/\idd{f_1,\ldots,f_n}$. Denote by $\overline{X}$ the affine scheme $\Spec(A\otimes k)$ obtained by the reduction modulo $\mathbf{m}$. Then, if the Jacobian matrix of $f_1,\ldots,f_n$ has determinant invertible then
$$
\#\overline{X}(k)<\#k^{n}
$$
where $\overline{X}(k)$ denotes the set of $k$-rational points of the scheme $\overline{X}$. Here, $\#k^n$ denotes the number of elements of $\field^n$
\end{unimodularII}

Motivated by the Unimodular Conjecture II we define %a new class of local domains: 
the classes of $d$-\textbf{unimodular domains} and \textbf{
  invariant domains} and we explore conditions where the Unimodular Conjecture II is true. A local domain $(R,\mathfrak{m},\field)$, with $\field$ finite, is called $d$-unimodular if given $f_1,\ldots,f_n \in R[x_1,\ldots,x_n]$ with jacobian $1$ and $\deg(f_i)\leq d$ \textbf{for some} $i$ we have $\#X(k)<\#\field^{n}$. In direction, we prove the following:

\begin{thmA} Let $R$ be a local domain with maximal ideal $\mathbf{m}$ and finite residue field $\field$. Let $f \in \mathcal{MP}_{n}(R)$ be a polynomial map with jacobian $1$. Then, $f$ is $(\#k-1)$ unimodular.
\end{thmA}

The interesting fact is that the Unimodular Conjecture is related to the Jacobian Conjecture( see \cite[Theorem 6]{poly2} and \cite[Theorem 4.5.10]{van2021jacobian}):

\begin{thm} \label{EssenL} The $p$-adic integer ring
$\mathbb{Z}_{p}$ does satisfy the Unimodular Conjecture for \textbf{almost all primes $p$}
if and only if the Jacobian Conjecture is true. 
\end{thm}

We recall that "\textbf{almost all primes $p$}" means all primes of $\mathbb{Z}$ except a finite number. The curious fact is that there is no known example of a prime $p$ such that $\mathbb{Z}_p$ satisfies the \textbf{Unimodular Conjecture II}. In this direction,  naturally appears the problem:

\begin{pro}\label{exunimod} Find a unimodular prime $p \in \mathbb{Z}$.
\end{pro}

The Problem \ref{exunimod} above motivates the following question:

\begin{pro} In the statement of the Unimodular Conjecture, can we replace "\textbf{almost all primes $p$}", with "\textbf{infinitely many primes $p$}", or "\textbf{a prime $p$}"? 
\end{pro}

We present some results in this direction. In particular, we prove that we can replace "\textbf{almost all primes $p$}", with "\textbf{a prime $p$}" (see Theorem \ref{main}) and so that the Jacobian Conjecture is equivalent to finding a prime $p$ such that $\mathbb{Z}_p$ is unimodular. We introduce the notion of Keller-finite domains and establish the following improvement of the Essen-Lipton Theorem using this notion.

\begin{thm} Let $p$ be a prime number. The Jacobian Conjecture is equivalent to the unimodularity of $\mathbb{Z}_p$.
\end{thm}

To prove the theorem above we will define and explore the notion of Keller-finite domains and we will prove that every complete discrete valuation ring is Keller-finite (see Theorem \ref{main}).

\subsection{Organization of the paper} In Section \ref{notacoes} we fix the notations that will be used in the paper. In Section \ref{introd} we survey the main properties of polynomial maps defined over domains. In Section \ref{uidomains} we define the unimodular and invariant domains and explore some properties.  In the last section, we define the notion of Keller-finite domains and we prove that every complete discrete valuation ring is a Keller-finite domain. We use this fact to give a new proof of the Essen-Lipton theorem.

\section{Notation} \label{notacoes}

\begin{itemize}

    \item $R = $ domain
    
    \item $\overline{\field} = $ an algebraic closure of a field $\field$
    
    \item $\left<a_1,\ldots, a_n \right> = $ ideal generated by $a_1,\ldots, a_n \in R$

    \item $\mathcal{Z}(f_1,\ldots, f_r) = $ affine scheme/algebric set given by polynomials $f_1,\ldots, f_r$
    
    \item given $a = (a_1,\ldots, a_n)$ and $b = (b_1,\ldots, b_n)$ in $R^n$:  $\idd{a-b} = \left<a_1-b_1,\ldots,a_n-b_n \right> \subset R$
    
    \item $(R,\mathfrak{m},\field) = $ local domain with maximal ideal $\mathfrak{m}$ and residue field $\field$
    
    \item $\mathcal{MP}_{n}(R) = $ polynomial maps over $R$ = collection of $(f_1,\ldots,f_n)$ with $f_i \in R[x_1,\ldots,x_n]$

    \item $\mbox{\textbf{\mbox{Aut}}}(R) = $ the collection of invertible polynomial maps over $R$

    \item $J_{f} = $ the Jacobian matrix associated to a polynomial map $(f_1,\ldots,f_n)$
    
    \item $\mathbb{Z}_p = $ the $p$-adic ring = the completion of $\mathbb{Z}$ at the maximal ideal $p\mathbb{Z}$
    
    \item $\mathbb{Q}_p = $ the fraction field of $\mathbb{Z}_p$

    \item If $S$ is a finite set, $\#S  = $ number of elements of $S$

    \item If $f,g\in \mathcal{MP}_{n}(R)$ then $f\circ g$ denotes the composition $(f_1(g_1,\ldots,g_n),\ldots, f_n(g_1,\ldots,g_n))$
    
\end{itemize}

\section{Polynomial maps}\label{introd}

In this section, we survey some properties of polynomial maps and present an introduction to the Jacobian Conjecture. Complete proofs can be found  \cite{poly0}.

Let $R$ be a domain. In this  paper, by a polynomial map over $R$ we mean a $n$-tuple of polynomial $f = (f_1,\ldots,f_n)$ with $f_i\in R[x_1,\ldots,x_n]$. We denote by $\mathcal{MP}_{n}(R)$ the collection of polynomial maps over $R$. We have an identification $\mathcal{MP}_{n}(R) \cong R[x_1,\ldots,x_n]^{n}$. If $f = (f_1,\ldots,f_n) \in \mathcal{MP}_{n}(R)$ is a polynomial map, we say that $f$ is \textbf{Keller map} if the jacobian matrix of $f$:
$$
J_{f} = \begin{bmatrix}
        \frac{\partial f_{1}}{\partial x_{1}}& \frac{\partial f_{1}}{\partial x_{2}} & \cdots & \frac{\partial f_{1}}{\partial x_{n-1}} & \frac{\partial f_{1}}{\partial x_{n}}  \\
        \cdot & \cdot & \cdots  & \cdot & \cdot     \\
        \cdot & \cdot & \cdots  & \cdot & \cdot     \\
        \cdot & \cdot & \cdots & \cdot & \cdot     \\
          \frac{\partial f_{n}}{\partial x_{1}}& \frac{\partial f_{n}}{\partial x_{2}} & \cdots & \frac{\partial f_{n}}{\partial x_{n-1}} & \frac{\partial f_{n}}{\partial x_{n}}  
        \end{bmatrix} 
$$
is invertible. This is equivalent to say that $\det J_{F} \in R[x_1,\ldots,x_n]^{*} = R^{*}$. We say that \textbf{$f$ has jacobian $1$} if $\det J_{f} = 1$. Given $f\in \mathcal{MP}_{n}(R)$ we say that $f$ is invertible if there is $g \in \mathcal{MP}(R)$ such that $f\circ g = id = g\circ f$, where $id$ is the polynomial map $f = (x_1,\ldots,x_n)$. By simple derivation it follows that if $f$ is invertible then $\det J_{f} \in R[x_{1},\ldots,x_n]^{*} = R^{*}$ and by normalization we can assume $\det J_{f} = 1$. 

The Jacobian Conjecture asks for the converse when $R = \mathbb{C}$.

\begin{jacobian} Any polynomial map $f = (f_1,\ldots,f_n) \in \mathcal{MP}_{n} (\mathbb{C})$ with jacobian $1$ is an invertible polynomial map.
\end{jacobian}

When $R$ has characteristic $p>0$ the analog statement is false. We have polynomial maps of the form: $f = (x_1^{p}-x_{1},\ldots,x_{n}^{n}-x_{n}) \in \mathcal{MP}_{n}(\mathbb{F}_p)$. Then, $f$ is not invertible, since is not 
injective. 

\begin{dfn} Let $f = (f_1,\ldots,f_n) \in \mathcal{MP}_{n}(R)$ be a polynomial map. The \textbf{degree of $f$} is the integer: $$\deg(f):= \textbf{Max}\{\deg(f_1),\ldots, \deg(f_n)\}.$$
\end{dfn}

In the paper \textbf{The Jacobian Conjecture: Reduction of degree and formal expansion of the inverse},  Hyman Bass, Edwin H. Connell and
David Wright explored the conjecture and proved, in particular, that the Jacobian Conjecture is equivalent to the following conjecture (see \cite[Corollary 2.2]{poly1}).

\begin{jacobianii} Any polynomial map $f = (f_1,\ldots,f_n) \in \mathcal{MP}_{n}(\mathbb{C})$ with jacobian $1$ \textbf{and $\deg(f) \leq 3$} is invertible.
\end{jacobianii}
In the \cite[Corollary 2.2]{poly1} it is proved more: we can take $f = (x_1+h_1,\ldots, x_n+h_n)$, where the jacobian matrix of $h = (h_1,\ldots,h_n)$ is nilpotent and $\deg(h_i) \leq 3$ for every $i$.

We denote by $\textbf{Aut}_{n}(R)$ the group of polynomial maps that are invertible. If  $f \in \mathcal{MP}_{n}(R)$ and $R \subset S$ for some domain  $S$, we can look $f$ as polynomial map over $S$. We denote this map by $f\otimes S$, obtained by scalar extension.

We recall some facts about polynomial maps.
 
\begin{prop} Let $f \in \mathcal{MP}_{n}(R)$ and $S$ be a domain with $R \subset S$. Then

$$f\otimes S \in \mbox{\textbf{\mbox{Aut}}}(S) \Longleftrightarrow f \in \mbox{\textbf{Aut}}(R).$$
 
\end{prop}

\begin{proof} see \cite[Lemma 1.1.8]{poly0}

\end{proof}

\begin{thm} \textbf{(Cynk-Rusek)} \label{CR} 
Fix an algebraically closed field $\field$  of characteristic $p\geq 0$. Let $X
  \subset \mathbb{A}_{\field}^{n}$ be an affine variety
and $f\colon X \longrightarrow X$ a regular map. The following 
conditions are equivalent:

(i) $f$ is injective;

(ii) $f$ is a bijection;

(iii) $f$ is an automorphism.

\end{thm}

\begin{proof} 
See \cite[Theorem 2.2]{poly16},\cite[Theorem 4.2.1]{poly0} or \cite[Theorem 3.1]{serre2009use} for more details. We will give proof of the implication $(i) \Longrightarrow (ii)$. Suppose by contradiction that there is a non-surjective and injective polynomial map $f\colon X\longrightarrow X$. 

\textbf{Step 1:} Formulate the conditions of non-surjectivity, injectivity, and membership using polynomial equations.
 \vspace{0.2cm}

 \textbf{Step 2:} Let $\{\alpha_i\}_{i\in I}$ the list of coefficients that appear on the equations of Step 1 and consider the cases:
 \begin{itemize}
  \item  $\ch(\corpo) = p>0$: Let $R = \mathbb{F}_p[\{\alpha_i\}_{i\in I}]$ the $\mathbb{F}_p$-algebra obtained by adjuntion of the all coefficients $\{\alpha_i\}_{i\in I}$. Take
  $\mathfrak{m}\in \spm(R)$ a maximal ideal of $R$. By the Nullstellensatz we conclude that  $R/\mathfrak{m}$ is a finite extension of $\mathbb{F}_p$. So, by reducing the polynomial relation above we get a polynomial map $$f\otimes R/\mathfrak{m}\colon  \overline{X}(R/\mathfrak{m}) \longrightarrow \overline{X}(R/\mathfrak
  {m})$$
that is injective and not surjective. But, this is a contraction since $\#X(R/\mathfrak{m})$ is a finite set.
\vspace{0.3cm}
 \item  $\ch(\corpo) = 0$: Let $R = \mathbb{Z}[\{\alpha_i\}]$ be the subring of $\corpo$ generated by the coefficients $\{\alpha_i\}_{i\in I}$. Let $\mathfrak{m} \in \spm(R)$ be a maximal ideal of $R$. Since $R/\mathfrak{m}$ is a finite field (see \cite[\href{https://stacks.math.columbia.edu/tag/00GC}{Lemma 00GC}]{stacks-project}) we reduce this case to the first case where $\ch(\corpo)>0$.
 \end{itemize}

\end{proof}

\begin{prop} Let $\field$ be an algebraically closed field with $2\nmid \textbf{char}(\field)$. Then any polynomial map $f = (f_1,\ldots, f_n) \in \mathcal{MP}_{n}(\field)$ of degree $d \in\{1,2\}$ with jacobian $1$ is invertible.
\end{prop}

\begin{proof} If $d = 1$ we can assume by a translation that $f_i$ is homogeneous of degree $1$ for every $i$. So, by linear algebra, $f$ is invertible if and only the Jacobian matrix is invertible. Now, assume $d = 2$. We follow the argument in \cite{poly1}. By the Theorem \ref{CR} it is sufficient to show that $f$ is injective. Suppose, by contradiction, that $f$ is not injective. We can assume that $(0,\ldots,0) = f(0,\ldots, 0) = f(h_1,\ldots, h_n)$ for some $h \in \mathbb{C}^{n}$ non-zero. Take $c = 1/2$ and write $f = f_1+f_2$ the homogeneous decomposition of $f$. Note that
$$
        0 =f_1(h)+2\cdot c\cdot f_2(h) = \frac{\partial[Tf_1(h)+T^{2}f_2(h)]}{\partial T}|_{T=c} = \frac{\partial f(Th)}{\partial T}|_{T=c}  = J_{f}(c\cdot h)\cdot h
    $$
a contradiction with the jacobian condition: $\det J_f = 1$. 
\end{proof}

\begin{lemma} \label{injeta} Let $\alpha_{1},...,\alpha_{n} \in \overline{\mathbb{Q}}$. 
Then for almost all primes $p \in \mathbb{Z}$
there is an injection of rings
$$\phi_{p}\colon \mathbb{Z}[\alpha_{1},...,\alpha_{n}] \hookrightarrow \mathbb{Z}_{p}.$$
\end{lemma}

\begin{proof} See \cite[Theorem 10.3.1]{poly0}
\end{proof}

\begin{Hlemma} Let $(R,\mathfrak{m},k)$ be a complete discrete valuation ring and 
$f_{1}, \ldots , f_{n} \in R[x_{1},\ldots ,x_{n}]$. Suppose that there is
$\alpha = (\alpha_{1},\dots ,\alpha_{n}) \in R^{n}$ such that
$$f_{1}(\alpha_{1},\dots ,\alpha_{n}) \equiv \cdots \equiv f_{n}(\alpha_{1},\dots ,\alpha_{n}) \equiv 0 
\mod \mathfrak{m}^{2m+1}$$
\noindent
where $m$ is the integer such that $\det J_f(\alpha) \in \mathfrak{m}^{m}-\mathfrak{m}^{m+1}$. Then there is a unique 
$\beta = (\beta_{1},\dots ,\beta_{n}) \in R^{n}$ such that 
$f_{1}(\beta) = \cdots = f_{n}(\beta) = 0$ and $\beta_{i} \equiv \alpha_{i} \mod \mathfrak{m}^{m+1}$
for all $i = 1,\dots ,n$.
\end{Hlemma}

\begin{proof} See \cite[proposition 5.20]{poly4}.
\end{proof}

By Hensel's lemma, we get the following

\begin{prop} \label{Ahensel} Let $(R,\mathfrak{m},\field)$ be a complete discrete valuation ring and let $f = (f_1,\ldots,f_n)$ be a
Keller map over $R$. If $A$ is an $R$-algebra denote by 
$X(A)$ the set of $A$-points of the affine scheme
$\textbf{\mbox{Spec}}(R[x_{1},\ldots ,x_{n}]/
\idd{f_{1},\dots ,f_{n}}
)$. Then there is a bijection $X(R) \cong X(k)$.
\end{prop}

\begin{proof} As $f$ is Keller map we have $Jf(u)$ invertible matrix for every $u = (u_1,\ldots,u_n) 
\in R^{n}$. The bijection is natural: given $u \in R^{n}$ we define $\varphi(u) \in X(k)$
the $k$-point obtained by reduction modulo $\mathfrak{m}$. The Hensel Lemma implies that the projection map $\pi \colon X(R) \longrightarrow X(k)$
is a bijection: injectivity by the uniqueness and
surjectivity by lifting points.   
\end{proof}

\section{Unimodular and invariant domains}\label{uidomains}

Let $(R,\mathfrak{m},\field)$ be a local ring. Given a polynomial map $f = (f_1,\ldots,f_n) \in \mathcal{MP}_{n}(R)$ we denote by $\overline{f} = f\otimes \field \in
\mathcal{MP}_{n}(\field)$ the induced map over the residue
field $\field$ via reduction modulo $\mathfrak{m}$.

\begin{dfn}  We say that 
$R$ is a \textbf{unimodular domain} if the following holds:
\begin{itemize}
    \item given $f_1,\ldots,f_n \in R[x_1,\ldots,x_n]$ with jacobian $1$ consider the scheme: $$X = \Spec(R[x_1,\ldots,x_n]/\idd{f_1,\ldots,f_n})$$
    Then $X(\field) \neq \mathbb{A}_{\field}^{n}$. 
\end{itemize}
We say that a polynomial map $f \in
\mathcal{MP}_{n}(R)$ is \textbf{unimodular} if it
satisfies the condition above.
\end{dfn}

\begin{OBS} If $f = (f_1,\ldots,f_n) \in \mathcal{MP}_n(R)$ is Keller polynomial map with $c = \det J_{f} \in R^{*}$ then $g = (c^{-1}f_1,f_2,\ldots,f_n)$ is a Keller map with jacobian $1$. In particular, if $(R,\mathfrak{m},\field)$ is an unimodular domain then any Keller polynomial map is a unimodular map.
 \end{OBS}

\begin{OBS} If $R$ is a complete discrete valuation ring with finite residue field then the unimodularity condition is equivalent to $\#X(R)<\#\field^{n}$. Indeed, by the Proposition \ref{Ahensel} we know that $X(R) = X(\field)$. In particular, $\#X(R) = \#X(\field)$.
\end{OBS}

\begin{prop}\label{casoinfinito} Let $(R,\mathfrak{m},\field)$  be a
 local domain with $\field$ an infinite field. Then $R$ is an unimodular domain.
\end{prop}

\begin{proof} Let $f \in \mathcal{MP}_{n}(R)$ be a Keller map. Let $\overline{f} = f\otimes \field \in \mathcal{MP}_{n}(\field)$  be the
induced map over the residue field and suppose that $f(\alpha) = 0$ for all 
$\alpha \in \field^{n}$. Since $\field$ is infinite we have $\overline{f} \equiv 0$. In particular, the coefficients in $f$ are in the maximal ideal $\mathfrak{m}$. 
In particular, $\det Jf \in\mathfrak{m}[x_{1},\ldots,x_{n}]$, a contradiction by jacobian condition: $\det J_f \in R^{*}$.
\end{proof}

We recall the following proposition (see \cite[Proposition 8]{poly2}).

\begin{prop}\label{UJC} Suppose that the Jacobian Conjecture is true.  Then every local domain $R$ of characteristic zero is unimodular.
\end{prop}

\begin{proof} Let $f \in \mathcal{MP}_{n}(R)$ be a Keller map defined over a local domain $R$ of characteristic zero. Since we assume the
Jacobian Conjecture is true over $\mathbb{C}$ 
we have $f$ an invertible map over $R$ (see \cite[lemma 1.1.14]{poly0}.
%\red{(Please explain, ref.)} 
In particular, there is 
$g \in \mathcal{MP}_{n}(R)$ such that $f\circ g
= x = (x_1,\ldots,x_n)$.  
By the reduction modulo $\mathfrak{m}$ we conclude that the map $\overline{f} \in \mathcal{MP}_{n}(\field)$ 
is a bijection, in particular, a non-zero map.
\end{proof}

The Unimodular Conjecture is false for local domains of characteristic $p > 0$ and finite residue field. 
  
\begin{ex} Consider the local domain $(\mathbb{F}_{p}[[t]],t\mathbb{F}_{p}[[t]],\mathbb{F}_{p})$ and 
take the polynomial map $f = (x_{1}-x_{1}^{p},\dots ,x_{n}-x_{n}^{p}) \in \mathcal{MP}_{n}(\mathbb{F}_{p}[[t]])$. Note that $f$ is a Keller map but the induced map over the residue field is the zero map because $\alpha^{p} = \alpha$ for every $\alpha \in \mathbb{F}_{p}$.
\end{ex}

\begin{OBS}
Let $(R,\mathfrak{m},\field)$  be a local domain. The following table shows the complete set of relations between $\ch(R)$ and $\ch(\field)$.

\begin{center}
    \begin{tabular}{|c|c|c|c| }
    \hline
    \textbf{$\ch(R)$} & \textbf{$\ch(\field)$} & $\#\field$ & $\mbox{\textbf{type}}$ \\ \hline
    $p = 0$ & $q > 0$  & $\mbox{$\infty$}$ & $\mbox{unimodular} $ \\ \hline
    $p = 0$ & $q > 0$ & $<\mbox{$\infty$}$ & $\mbox{\textbf{unknown}}$ \\  \hline
    $p = 0$ & $q = 0$ & $ \mbox{$\infty$}$ & $\mbox{unimodular} $ \\ \hline
    $p>0$ & $q = p$ & $<\mbox{$\infty$}$  & $\mbox{non-unimodular} $\\ \hline
    $p>0$ & $q = p$ & $\mbox{$\infty$}$  & $\mbox{unimodular} $\\ 
    \hline
    \end{tabular}
\end{center}
\end{OBS}

\subsection{Invariance and unimodularity} Now we will study when the unimodularity of a polynomial map is preserved when we make operations and composition with the map.

\begin{dfn} \label{invvv} Let $R$  be a local domain with of characteristic zero and $f \in \mathcal{MP}_{n}(R)$ be a unimodular map. We say that $f$ is an \textbf{invariant map} if it satisfies the following:
\begin{itemize}
    \item Let $a\in R^{n}$ and $g \in \mathcal{MP}_{n}(R)$ 
be a \textbf{Keller affine automorphism}, that is, $g = AX+b$ with
$$A = \begin{bmatrix}
        a_{11}& a_{12} & \cdots & a_{1(n-1)} & a_{1n}  \\
        \cdot & \cdot & \cdots  & \cdot & \cdot     \\
        \cdot & \cdot & \cdots  & \cdot & \cdot     \\
        \cdot & \cdot & \cdots & \cdot & \cdot     \\
          a_{n1}& a_{n2} & \cdots & a_{n(n-1)} & a_{nn}  
        \end{bmatrix}$$
having $\det A = 1$ and $b =(b_1,\ldots,b_n)$.  Then $f\circ g\circ f$ and $f-f(a) = (f_1-f_1(a),\ldots,f_n-f_n(a))$ are unimodular maps for every $a \in R^{n}$.
\end{itemize}

\end{dfn}

Note that in the Definition \ref{invvv} we ask the unimodular property to be invariant under translation and 
composition of a special type. Note also that, as in the unimodular case, if the residue field $\field$ is infinite then any Keller unimodular map is an invariant map.
  
\begin{prop}\label{uinvariant} Let $R$  be a local unimodular domain. Then $R$ is invariant.
\end{prop}

\begin{proof} By hypothesis a Keller map $f \in
  \mathcal{MP}_{n}(R)$ is
  unimodular. Since the Keller condition is 
invariant under composition and translation we have the result.
\end{proof}

\begin{dfn} Let $R$ be a local domain. 
Given a map $f \in \mathcal{MP}_{n}(R)$ we say that 
\begin{itemize}

 \item $f$ is \textbf{strongly invariant} if it is invariant and for all \textbf{Keller affine automorphisms}   
$g_{1},\dots ,g_{k} \in \mathcal{MP}_{n}(R)$ the map $f_{1}\circ f_{2} \circ f_{3}\circ \cdots \circ f_{k}$ is
invariant where $f_{j} = g_{j}\circ f$.

\end{itemize}

The domain $R$ is called an \textbf{invariant domain} if every unimodular polynomial map (in dimension $n > 1$) is invariant. 
\end{dfn}

\begin{lemma} If a map $f \in \mathcal{MP}_{n}(R)$ is strongly invariant then $f\circ g\circ f$ is strongly invariant for all Keller affine automorphism $g \in \mathcal{MP}_{n}(R)$.
\end{lemma}

\begin{proof} Indeed, by induction, it is sufficient to consider the case $k = 2$. For this, 
let $g_{1},g_{2} \in \mathcal{MP}_{n}(R)$ be a Keller affine automorphism and observe that 
$g_{1}\circ (f\circ g\circ f)\circ g_{2}\circ (f\circ g\circ f) = 
(g_{1}\circ f)\circ (g\circ f)\circ (g_{2}\circ f)\circ (g\circ f)$. So it is invariant by hypothesis on $f$.
\end{proof}
  
The next example shows that the condition $\ch(R) = 0$ is important.

\begin{ex} Let $f_{1},\dots ,f_{n} \in
  \mathbb{F}_{p}[[t]][x_{1},\dots ,x_{n}]$ be defined by 
$f_{j} = 1-x_{j}^{p}+x_{j}$ and consider the polynomial map $f = (f_{1},\dots ,f_{n}) \in \mathcal{MP}_{n}(\mathbb{F}_{p}[[t]])$. It is easy to check that $\det J_f = 1$ and that $f$ is an unimodular map.  But,
$$f - f(1,\dots ,1) = (-x_{1}^{p}+x_{1},\dots ,-x_{n}^{p}+x_{n}).$$ 
So, in the case $R = \mathbb{F}_p[[t]]$ it follows that the property of invariance by translation is false. 
\end{ex}
  
\begin{ex} Let $g(x) \in \mathbb{F}_{p}[x]$  be a polynomial that maps $\{0,\dots ,p-2\} \mapsto p-1$ and $p-1 \mapsto 0$. 
For example, take $p = 5$ and consider 
$$g(x) = -1+x-x^{2}+x^{3}-x^{4} \in \mathbb{F}_{5}[x].$$
It is easy to check that $g\circ g = 0$. Note that $g(0) \neq 0$.  Define the polynomial map 
$f = (f_{1},\dots ,f_{n}) \in \mathcal{MP}_{n}(\mathbb{F}_{p}[[t]])$ with $f_{j} = x_{j}-x_{j}^{p}+g(x_{j}^{p})$. 
We have $f$ a Keller map with the induced map over the residue field non-zero. But by construction, we have $f\circ f = 0$. 
Thus, in characteristic $p > 0$ the invariance by composition is false. 
\end{ex}

In the next theorem, the argument is similar to the argument given in \cite[Theorem 4]{poly2} with the observation that it is sufficient to require the invariance property.

\begin{thm} \label{fiber2} Let
  $(R,\mathfrak{m},\field)$
 be a complete discrete valuation ring with finite
 residue field $\field$. Let $f \in \mathcal{MP}_{n}(R)$  be a strongly invariant map. Then $f$ is injective.
 \end{thm}

\begin{proof} Suppose, by contradiction that there is a a strongly
  invariant polynomial map $f = (f_1,\ldots,f_n) \in \mathcal{MP}_{n}(R)$
with $f(a_{1}) = \cdots = f(a_{m}) = c$ ($m>1$) for some $a_{1},\dots ,a_{m} \in R^n$ with
$a_{i} \neq a_{j}$, if $i\neq j$. We will show that there is a strongly invariant map $g$ with $\#g^{-1}(c) > m$. 
By iteration we will get a Keller map $\widetilde{g} \in \mathcal{MP}_{n}(R)$ with $\#\widetilde{g}^{-1}(c) > (\#k)^{n}$ 
a contradiction by Proposition \ref{Ahensel}.

Since $f(a_{1}) = f(a_{2})$ we have $\idd{a_{2}-a_{1}}
= R$ by the \cite[Lemma 10.3.11]{poly0}). On the other hand, since
$f$ is an invariant map we guarantee that there exists $b \in R^{n}$ 
such that $f(b)-f(a_{1})$ is unimodular, that is, 
$\idd{f(b)-f(a_{1}} = R$. In particular,
$\idd{a_{2}-a_{1}} = \idd{f(b)-f(a_{1})}= \idd{f(b)-c} = R$. So, we have 
$\{a_{2},a_{1}\} \cong \{f(b),c\}$ (see \cite[Transitivity, Proposition 1]{poly2}). 
By \cite[Theorem 2]{poly2} we know that there is $h \in \mathcal{MP}_{n}(R)$, Keller
affine automorphism such that $h(c) = a_{1}$ and $h(f(b)) = a_{2}$. Now define $g = f\circ h\circ f$.  Then the map $g$ is strongly invariant map with  $g(a_{j}) = f(h(c)) = f(a_{1}) = c$ for all $j$ and
$h(b) = f(h(f(b))) =f(a_{2}) = c$.  Note that $b \neq
a_{j}$ for all $j$.
\end{proof}

As a consequence, we get the following result.

\begin{cor}\label{fiber3} Let $(R,\mathfrak{m},\field)$  be a complete discrete valuation ring with 
a finite residue field $\field$. Suppose that $R$ is an invariant domain. Then any unimodular Keller polynomial map 
$f \in \mathcal{MP}_{n}(R)$ is an injective map.\end{cor}

\begin{dfn} Pick $d \in \mathbb{Z}_{\geq 1}$ and let $(R,\mathfrak{m},\field)$ be a local domain. We say that $R$ is a $d$-unimodular map if any Keller map $f = (f_1, \ldots, f_n)\in \mathcal{MP}_{n}(R)$ with $\deg(f_i)\leq d$, \textbf{for some $i$}, is unimodular.
\end{dfn}

Note that any local domain $R$ is
$1$-unimodular and $R$ is a 
unimodular domain if and only if it 
is $d$-unimodular for all $d \in \mathbb{N}$. 
If $R$ is $d$-unimodular then it is $e$-unimodular
for all $e \leq d$. 

If $R$ has characteristic $p>0$ and $\field$ is finite then  $R$ is not $d$-unimodular 
for infinitely many $d \in \mathbb{Z}$. Indeed, for each 
$m \in \mathbb{N}$ take $d = (\#k)^{m}$ and consider the map $f = (x_{1}-x_{1}^{d},\dots ,x_{n}-x_{n}^{d}) \in \mathcal{MP}_{n}(R)$.

\begin{prop} \label{menor} Let $f \in \mathcal{MP}_{n}(\mathbb{Z})$ be a non-constant polynomial map. Then for almost all primes $p \in \mathbb{Z}$ we have $F\otimes \mathbb{Z}_{p}$ unimodular map over $\mathbb{Z}_{p}$.
\end{prop}

\begin{proof} Indeed, suppose $f_{1}(x_{1},\dots ,x_{n}) \in \mathbb{Z}[x_{1},\dots ,x_{n}] \setminus \mathbb{Z}$.
We can choose $d \in \mathbb{Z}^{n}$ such that $f_{1}(d) \neq 0$.  Note that $f_{1}(d) \in \mathbb{Z}_{p}^{*}$ for all 
$p$ such that $p \nmid f_{1}(d)$.
\end{proof}
  
It is known that to prove the Jacobian Conjecture it is sufficient to consider polynomial maps of Druzkowski type, that is, maps in the form $f = x+h$ with $h_{j} = (\sum_{k}a_{kj}x_{k})^{3}$ and $Jh$ nilpotent (see \cite[Theorem 6.3.2]{poly0}).  We call maps of the form $f = x+h$ with $h = \sum_{k}a_{kj}x_{k}^{3}$ quasi-Druzkowski maps.

\begin{prop} Quasi-Druzkowski maps are unimodular over $\mathbb{Z}_p$.
\end{prop}

\begin{proof} Let $f$ be a quasi-Druzkowski map with $h = (h_{1},\dots ,h_{n})$ where $h_{j} =
\sum_{k}b_{kj}x_{k}^{3}$.  We will show that there exist $u_{1},\dots ,u_{n} \in \mathbb{Z}_{p}$ non-zero such that

$$u_{1}h_{1}(x_{1},\dots ,x_{n})+\cdots+u_{n}h_{n}(x_{1},\dots ,x_{n}) = 0.$$

Indeed, for this, it is sufficient to find a non-trivial solution for the homogeneous system:
$$u_{1}b_{11}+u_{2}b_{12}+\cdots+u_{n}b_{1n} = u_{1}b_{21}+u_{2}b_{22}+\cdots+u_{n}b_{2n} = \cdots =u_{1}b_{n1}+u_{2}b_{n2}
+\cdots+u_{n}b_{nn} = 0.$$
Now since $JH$ is nilpotent we have, in particular, $\det(b_{ij}) = 0$ and so  there is a non-trivial
solution
$(u_{1},\dots ,u_{n})\in \mathbb{Q}_{p}^{n}$ 
for the system above. Without loss of generality we can suppose that 
$u_{1}\in \mathbb{Z}_{p}^{*}$ and $u_{j} \in \mathbb{Z}_{p}$, if $j > 1$. Now consider
$s = u_{1}+u_{2}p\cdots+u_{n}p \in
\mathbb{Z}_{p}^{*}$. Note that, $(1,p,\dots ,p) \in
\mathbb{Z}_{p}^{n}$ is such that $(f_{1}(1,p,\dots ,p),\dots ,f_{n}(1,p,\dots ,p)) = \mathbb{Z}_{p}.$
\end{proof}

\begin{OBS}
It was seen in the previous section that there are local domains $(\mathcal{O}, \mathcal{M},k)$ 
with characteristic $p>0$ that are not unimodular
domains. On the other hand, we know that any local domain with an infinite residue field is indeed 
an unimodular domain. In particular, 
if we consider the map $f = (x_{1}-x_{1}^{p},\dots ,x_{n}-x_{n}^{p})$ over  
$(\overline{\mathbb{F}}_{p}[[T]],T\overline{\mathbb{F}}_{p}[[T]],\overline{\mathbb{F}}_{p})$  
we have $\overline{F}(\alpha) \neq 0$ for some 
$\alpha \in \overline{\mathbb{F}}_{p}$.
So, if we take, $L$, the field obtained by adjunction of $\alpha$ to $\mathbb{F}_{p}$ we see that our 
$f$ is unimodular over the local domain $(L[[T]],TL[[T]], L)$. 
\end{OBS}

For the $p$-adic case, there is an analog:

\begin{thm} \label{engordando} Let $f \in \mathcal{MP}_{n}(\mathbb{Z}_{p})$  be a Keller polynomial map. Then
there is a complete discrete valuation ring $(\mathcal{O},\mathcal{M}, \field)$ that dominates $\mathbb{Z}_{p}$ 
such that $f\otimes \mathcal{O}$ is a unimodular map. Furthermore, $\mathcal{O}$ is a free $\mathbb{Z}_{p}$-module  with $\textbf{\mbox{rank}}_{\mathbb{Z}_{p}}(\mathcal{O}) = [\field:\mathbb{F}_{p}].$
\end{thm}

\begin{proof}  Consider the map $\overline{f} \in \mathcal{MP}_{n}(\mathbb{F}_{p})$ induced over the residue field. By the previous remark we know that there exists $\alpha \in \overline{\mathbb{F}_{p}}^{n}$ such that $\overline{f}(\alpha) \neq 0$. By taking the field $\field = \mathbb{F}_{p}(\alpha_{1},\dots ,\alpha_{n})$ 
obtained by adjunction we can look $\overline{f}$ as a
polynomial map which is \textbf{non zero} over $\field$. 
Now we recall the following theorem about unramified extensions of a local field $L$ (\cite[Proposition 7.50]{poly6})

\begin{thmo} Let $L$  be a local field with residue field $l$. There exists a 1-1 correspondence
between the following sets
$$\mbox{ $\{$finite extensions  unramified of $L$ $\}$ } \cong \mbox{ $\{$finite extensions of $l$$\}$}$$
\noindent
given by $L' \mapsto l'$, where $l'$ is the residue field associated to $L'$.  Furthermore, in this correspondence,
we have $[L':L] = [l':l]$.
\end{thmo}
Applying the theorem above to $L = \mathbb{Q}_{p}$  with
$l = \mathbb{F}_{p}$ we see that the extension
$\field|\mathbb{F}_{p}$ corresponds to a local field $K|\mathbb{Q}_{p}$ such that $\field$ is the residue field of $K$. 
Denote by $(\mathcal{O}, \mathcal{M},k)$ the ring of integers of $K$. 
The ring $\mathcal{O}$ is the integral closure of $\mathbb{Z}_{p}$ in $K$ and by 
\cite[Proposition 5.17]{poly7}) we know 
that $\mathcal{O}$ is a free $\mathbb{Z}_{p}$-module and $rank_{\mathbb{Z}_{p}}(\mathcal{O}) = [K:\mathbb{Q}_{p}] = 
[\field:\mathbb{F}_{p}]$. So $f\otimes \mathcal{O} \in \mathcal{MP}_{n}(\mathcal{O})$ is a Keller map with a non-zero induced map over the residue field. 
\end{proof}

\begin{lemma}\label{decendo3} Let $K|\mathbb{Q}_{p}$ be a finite Galois extension with $ m = [K:\mathbb{Q}_{p}]>1$. Let
$\mathcal{O}_{K}$ be the integral closure of $\mathbb{Z}$ in
$K$. Let $f \in \mathcal{MP}_{n}(\mathcal{O}_{K})$ be a
 non-injective Keller unimodular map. Then there exists a non-injective Keller unimodular map $g \in \mathcal{MP}_{mn}(\mathbb{Z}_{p})$.
\end{lemma}

\begin{proof} The same argument of \cite[A Galois descent]{poly0} works. The relevant fact is that $\mathcal{O}_{K}$
is a free $\mathbb{Z}_{p}$-module of $\textbf{\mbox{rank}}_{\mathbb{Z}_{p}}(\mathcal{O}_{K}) = m$.
\end{proof}

%The next theorem relates to the following conjecture (see \cite{}). 

%\begin{jacobianii} Denote by $\mathcal{O}$ the integral closure of $\mathbb{Z}$ in $\overline{\mathbb{Q}}$. 
%Let $f\colon \mathbb{Z}^{n} \longrightarrow \mathbb{Z}^{n}$ be a  Keller map such that 
%$f\otimes\mathcal{O}$ is injective. Then $f$ is an isomorphism.
%\end{jacobianii}

By Theorem \ref{fiber2} we know that if $\mathbb{Z}_{p}$ is an unimodular domain then any Keller map $f \in \mathcal{MP}_n(\mathbb{Z}_{p})$ is injective. We can show a more general result

\begin{thm} Assume that $\mathbb{Z}_{p}$ is an invariant domain for some prime $p$. Then for all Keller unimodular maps $f \in \mathcal{MP}_{n}(\mathbb{Z}_{p})$ and $K|\mathbb{Q}_{p}$ finite extension we have $f\otimes \mathcal{O}_{K}$ is an injective map.
\end{thm}

\begin{proof} It is sufficient to show that $f\otimes \mathcal{O}$ is an injective map where 
$\mathcal{O}$ denotes the integral closure of $\mathbb{Z}_{p}$ in $\overline{\mathbb{Q}_{p}}$. 
Indeed, if $\alpha \neq \beta \in \mathcal{O}^n$ are such that $f(\alpha) = f(\beta)$ consider
the ring $R = \mathbb{Z}_{p}[\alpha,\beta]$ obtained by adjunction of $\alpha$ and $\beta$ and let $K$ the fraction field of $R$. 
The extension $K|\mathbb{Q}_{p}$ is finite and $\alpha_{1},\dots ,\alpha_{n},\beta_{1},\dots ,\beta_{n} \in K$. Note that 
$\alpha_{i},\beta_{i} \in \mathcal{O}_{K}$ for all $i$. Without loss of generality, we can suppose that $K|\mathbb{Q}_{p}$ is a Galois extension. So, $f\otimes \mathcal{O}_{K}$ is a Keller unimodular map over $\mathcal{O}_{K}$ and non-injective. By the Lemma \ref{decendo3}, we get $g \in \mathcal{MP}_{N}(\mathbb{Z}_{p})$ a Keller unimodular 
map non-injective, for some $N\in\mathbb{N}$. Now, since we are assuming that $\mathbb{Z}_{p}$ is an
invariant domain and by Theorem \autoref{fiber2} we know that $g$ is an injective map. A contradiction. 
\end{proof}

\subsection{$\mathbb{F}_p$-points of hypersurfaces and unimodularity}

We start this subsection by reminding the following result about $\mathbb{F}_p$-points of affine hypersurfaces in the affine space $\mathbb{A}_{\overline{\mathbb{F}}_p}^{n}$ (see \cite[Corollary 2.7]{ghorpade2019note})

\begin{thm} Let $X = \mathcal{Z}(f) \subset \mathbb{A}_{\overline{\mathbb{F}}_p}^{n}$ be an affine hypersurface. Consider the set of $\mathbb{F}^{n}$-points of $X$, that is, 
$$
X(\mathbb{F}_p) = \{(\alpha_1,\ldots,\alpha_n) \in \mathbb{A}_{\overline{\mathbb{F}}_p}^{n}\mid f(\alpha_1, \ldots,\alpha_n) = 0 \}\cap \mathbb{F}_p^{n}.
$$
Then $\#X(\mathbb{F}_p)\leq \deg(f)p^{n-1}$.
\end{thm}

Using this result we can easily proof that the $p$-adic ring $\mathbb{Z}_{p}$ is $(p-1)$-unimodular. 

\begin{prop}\label{fino} Let $(R, \mathfrak{m},\field)$ be a complete discrete valuation ring with residue field $\field$ finite. Then, any Keller polynomial map $f = (f_1,\ldots, f_n) \in \mathcal{MP}_{n}(R)$ is $(\#\field-1)$-unimodular if $\deg(f_i)<\#k$, for some $i$.
\end{prop}

\begin{proof} Let $X$ be the scheme defined by the equations $f_1,\ldots,f_n.$ First, consider the algebraic set $\overline{X} = \{(x_1,\ldots,x_n) \in \mathbb{A}_{\overline{\mathbb{F}}_p}^{n}\mid \overline{f}_1(x) = \cdots = \overline{f}_n(x) = 0\}$. Note that the hypersurface $\mathcal{Z}(f_i) \subset \mathbb{A}_{\overline{\mathbb{F}}_p}^{n}$ contains $\overline{X}$. In particular, $\#\overline{X}(\mathbb{F}_p)\leq \#\mathcal{Z}(f_i)(\field) \leq \deg(f_i)\#\field^{n-1}$. Now, by the Proposition \ref{Ahensel} we conclude that 
$$
\#X(R) = \#\overline{X}(\field) \leq \deg(f_i)\#\field^{n-1}<\#\field^{n}.
$$
This finishes the proof.
\end{proof}

\begin{cor} For all prime $p$,  
$\mathbb{F}_{p}[[T]]$ and $\mathbb{Z}_{p}$ are
$(p-1)$-unimodular domains .
\end{cor}

Note that the bound $p-1$ is ``maximal'' for $\mathbb{F}_{p}[[T]]$.
  
\begin{prop} 
Let $p \in \mathbb{Z}$ be a prime. For each $d \in \mathbb{Z}_{\geq 1}$ we can find a finite extension  $K|\mathbb{Q}_{p}$ such that the ring of integers
$\mathcal{O}_{K}$ is a $d$-unimodular domain.
\end{prop}

\begin{proof} Let $d \in \mathbb{N}$. If $d = 1$ just take $K = \mathbb{Q}_{p}$.
Suppose that $d > 1$. We know that for any Keller polynomial map $f  \in \mathcal{MP}_{n}(\mathbb{Z}_{p})$ of degree $d$ we have 
$$\#\mathcal{Z}(f_1,\dots,f_n) \leq \deg(f)^{n} = d^{n}$$
 \noindent
where $\mathcal{Z}(f_1,\dots,f_n)$ is the algebraic set in
$\mathbb{A}_{\overline{\mathbb{F}}_{p}}^{n}$ given by reduction of $f \mod p$. Let $n$ be an integer such that $p^{n} > d$ and fix  $\mathbb{F}_{p^{n}}$ the unique extension of $\mathbb{F}_{p}$ of degree
$n$ in $\overline{\mathbb{F}}_{p}$. We have seen in the proof of Theorem \autoref{engordando}  that there is a finite extension $K|\mathbb{Q}_{p}$ such that the residue field of $\mathcal{O}_{K}$ is  $\mathbb{F}_{p^{n}}$. By construction, for all Keller map $g \in \mathcal{MP}_{n}(\mathcal{O}_{K})$ with $\deg(g) \leq d$ we have $\#\{g_{1} = \cdots = g_{n} = 0\} \leq 
\deg(g)^{n}\leq d^{n} < (p^{n})^{n}$. So $g$ is an unimodular map and $\mathcal{O}_{K}$ is a $d$-invariant domain.
\end{proof}

\begin{prop} Suppose that for all $n \in
  \mathbb{N}$ and all 
$f = (f_{1},\dots ,f_{n}) \in \mathcal{MP}_{n}(\mathbb{Z}_{p})$ Keller map with
$\deg(f) < n$ is unimodular. Then $\mathbb{Z}_{p}$ is an unimodular domain.
\end{prop}

\begin{proof} Let $f \in
  \mathcal{MP}_{n}(\mathbb{Z}_{p})$ be a Keller map with 
$n \leq \deg(f)$. Let $m \in \mathbb{Z}$ be an integer (to be determined) and consider the map 
$$f^{[[m]]} = (f_{1},\dots ,f_{n},f_{1},\dots ,f_{n},\dots ,f_{1},\dots ,f_{n}) \in \mathcal{MP}_{mn}(\mathbb{Z}_{p})$$
\noindent
which consists of $m$-repetitions of the tuple $f_{1},\dots ,f_{n}$ where each occurrence of such tuple we introduce
 $n$-distinct variables. By construction we have
 $f^{[[m]]}$ a Keller map and $f$ is a unimodular map if
 and only if  so is $f^{[[m]]}$. We can choose large $m$ such that
 $\deg(f) < mn$. Thus, we get the unimodularity of $f$.
\end{proof}

%\ \begin{comment}\begin{cor} Let $f_1,\ldots,f_n \in \mathbb{Z}_{p}$ with the jacobian $1$ and suppose that $p>n$. Then, $f$ is  unimodular. \end{cor} \begin{proof} By the Theorem ? we can assume that $\deg f_i<n$ for all $i$. In that case, by the Ore inequality, we obtain: $$ X(\mathbb{Z}_p) = X(\mathbb{F}_p)\leq \deg(f_i)p^{n-1}<np^{n-1}<p^{n}. $$ So, $f$ is unimodular. \end{proof}\end{comment}

Let $R$ be a domain and $f \in R[x_{1},\dots ,x_{n}]$. Define $d(f):=$ the number of monomials in degree $>3$ that occur in  $f$. If $f = (f_{1},\dots ,f_{n}) \in \mathcal{MP}_{n}(R)$ we define $d(F) := \sum_{j}d(f_{j})$.

\begin{prop} Let $p \in \mathbb{Z}_{>3}$ be a prime number and $f \in \mathcal{MP}_{n}(\mathbb{Z}_{p})$ a Keller map. 
Suppose that 
$$d(f) \leq log(2)^{-1}log(nlog(p/3)/log(3))\qquad (*)$$
where $log$ is the natural logarithm. Then $f$ is unimodular. 
\end{prop}

\begin{proof} Let $f \in
  \mathcal{MP}_{n}(\mathbb{Z}_{p})$ be a Keller map. By the
 Reduction Theorem (see \cite[(Proposition 3.1]{poly1} we can find invertible
maps $g, h \in \mathcal{MP}_{n+m}(\mathbb{Z}_{p})$ for  some $m \in \mathbb{N}$ such that $G := g\circ f^{[m]}\circ h$
has degree $\leq 3$ where $f^{[m]} = (f,x_{n+1},\dots ,x_{n+m})$. Furthermore, we know that $g(0) = h(0) = 0$. Denote by 
$X_{f}(\mathbb{Z}_{p})$ and $X_{G}(\mathbb{Z}_{p})$ the  set of $\mathbb{Z}_{p}$-points of $f$ and $G$ respectively.
It is easy to check that $\#X_{f}(\mathbb{Z}_{p}) =   \#X_{G}(\mathbb{Z}_{p})$. Now, since $\mathbb{Z}_{p}$
is a $3$-unimodular domain we have
$\#X_{G}(\mathbb{Z}_{p}) < 3^{n+m}$ and we get $m = 2^{d(f)}$ by the proof of reduction theorem. The inequality $(*)$ implies $3^{m+n}\leq p^{n}$ and so we have $f$ a unimodular map.
 \end{proof}

\begin{thm}$\mathbb{Z}_{p}$ is an invariant domain for almost all prime $p$ if and only if the Jacobian Conjecture is true.
\end{thm}

\begin{proof} The implication $ \Longleftarrow$ it follows
  from Proposition \ref{UJC}. Suppose that $\mathbb{Z}_{p}$ is invariant for almost all prime $p$. By \cite[Proposition 1.1.19]{poly0} we know that
it is sufficient to show that the Jacobian Conjecture is true over $\mathbb{Z}$. Suppose, by contradiction, that there is some Keller map $f = (f_1,\ldots, f_n) \in \mathcal{MP}_{n}(\mathbb{Z})$  non-invertible. Since $f$ has
coefficients in $\mathbb{Z} \subset \mathbb{Z}_p$ it follows that $g$ is unimodular over $\mathbb{Z}_{p}$ for almost all primes $p$ (see Proposition \ref{menor}). Also, we 
know by the hypothesis that  $f\otimes
\overline{\mathbb{Q}}$ is non-injective (see Theorem \ref{CR}). In particular, by Lemma \ref{injeta}, we have $f$ non-injective over 
$\mathbb{Z}_{p}$ for  infinitely many primes $p$. Fix a prime $p$ such that $\mathbb{Z}_{p}$ is an 
invariant domain. So, we obtain $f \otimes \mathbb{Z}_{p}$ a Keller map non-injective over the invariant complete local domain. Contradiction by  Theorem \autoref{fiber2}.
\end{proof}

A refinement of Lemma \ref{injeta} allows us to replace \textbf{infinitely many primes} by \textbf{almost all primes}. This is the following lemma.

\begin{lemma} \label{injeta2}Let $\alpha_{1},\dots ,\alpha_{m} \in
  \overline{\mathbb{Q}}$ be algebraic numbers. Then there
  is a finite set  
$E$ of rational primes such that for all prime $p \notin
E$ we have an injective homomorphism 
  
$$\mathbb{Z}[\alpha_{1},\dots ,\alpha_{m}] \hookrightarrow \mathcal{O}_{K,p}$$ 

\noindent
where $\mathcal{O}_{K,p}$ is the ring of integers of some finite $K|\mathbb{Q}_{p}$.
\end{lemma}

\begin{proof} It is sufficient to prove the following
\textbf{Fact.} Let $f(T) \in \mathbb{Z}[T]\setminus{}\mathbb{Z}$ be an irreducible polynomial. Then for almost all prime $p$ there is  a finite extension $K|\mathbb{Q}_{p}$ and $\alpha \in \mathcal{O}_{K,p}$ such that $f(\alpha) = 0$.

Let $d$ be the discriminant of the polynomial $f$ and $E := \{ p \mid p\mbox{ is prime with $p\mid d$}\}$. 
Let $p \in \mathbb{Z}\setminus
E$ be a prime and take $\overline{f}(T) \in \mathbb{F}_{p}[T]$, via reduction $\mod p$. 
Let $\alpha \in \overline{\mathbb{F}_{p}}$  be a root of $\overline{f}(T)$ and take $\mathbb{F}_{p^{k}}$ 
the definition field of $\alpha$. Then

$$\overline{f}(\alpha) = 0 \mbox{ and } \overline{f}'(\alpha) \neq 0 \mbox{ by condition $p\notin E$}.$$

Now we recall that there exists a finite extension $K|\mathbb{Q}_{p}$ such that $\mathcal{O}_{K,p}$ is a 
complete discrete valuation with residue field
$\mathbb{F}_{p^{k}}$. Since $\mathcal{O}_{K,p}$ is a complete 
ring we can use the Hensel lemma to conclude that there is some $a \in \mathcal{O}_{K,p}$ such that 
$f(a) = 0$.
\end{proof}

\begin{thm}\label{refinamento} $\mathbb{Z}_{p}$ is an
  invariant domain for infinitely many primes $p$ if and
  only if the Jacobian Conjecture  is
  true. 

\end{thm}

\begin{proof} The implication $ \Longleftarrow$ is
  trivial. Suppose, by contradiction, that $\mathbb{Z}_{p}$ is an invariant domain 
for infinitely many primes $p$ but the Jacobian Conjecture is false. Let $f \in \mathcal{MP}_{N}(\mathbb{Z})$  be a
counterexample with  $\det JF =1$ (see \cite[Proposition 1.1.19]{poly0}). In particular, $f\otimes\overline{\mathbb{Q}}$ is not injective.  Let $\alpha \neq \beta \in \overline{\mathbb{Q}}^N$  be such that 
$f(\alpha) = f(\beta)$. By the Lemma \ref{injeta2}, we know that
$R = \mathbb{Z}[\alpha,\beta] \hookrightarrow \mathcal{O}_{K,p}$ 
for almost all primes $p$. Fix a prime $p$ such that $R \hookrightarrow \mathcal{O}_{K,p}$ and such that $\mathbb{Z}_{p}$
is an invariant domain. So, we obtain $f\otimes\mathcal{O}_{K,p} $ a Keller map, non-injective over the domain 
$\mathcal{O}_{K,p}$. By Lemma \ref{decendo3} we know that there exists a Keller map $g$ over $\mathbb{Z}_{p}$ 
that non-injective. A contradiction by Theorem \autoref{fiber3}.
\end{proof}

As a consequence, we get the following interesting result.

\begin{cor}\label{equivalencia} There is a finite set of
  primes $E$ such that for all prime $p \in
  \mathbb{Z}\setminus{}
E$ we have
$$\mathbb{Z}_{p}\mbox{ is an invariant domain}
\Longleftrightarrow \mathbb{Z}_{p} \mbox{ is a unimodular domain}.$$
\end{cor}

\begin{proof} The implication ($\Longleftarrow$) follows from Proposition \ref{uinvariant}. Suppose ($\Longrightarrow$) is false. Then for infinitely many primes $p$, we have $\mathbb{Z}_{p}$ as an invariant and non-unimodular domain. Since $\mathbb{Z}_{p}$ is invariant for infinitely many primes we have that the Jacobian Conjecture is true by Theorem  
\autoref{refinamento}. Contradiction by
Essen-Lipton theorem.
\end{proof}

\section{Keller-finite domains}

This section introduces the notion of Keller-finite domains and explores their relation with the Jacobian Conjecture. Using this notion we give a simple proof of a result that refines the Essen-Lipton Theorem in some sense (see Theorem \ref{main}).

We start with the main definition.

\begin{dfn} Let $R$ be a domain and $n \in \mathbb{N}$. We say that $R$ is a \textbf{Keller-finite domain in dimension $n$} if it satisfies the following property:

\begin{itemize}
    \item given $f_1,\ldots,f_n \in R[x_1,\ldots,x_n]$ with jacobian $1$ the $R$-module $$R[x_1,\ldots,x_n]/\idd{f_1,\ldots,f_n}$$ is finitely generated.
\end{itemize}
We say that $R$ is a \textbf{Keller-finite domain} if $R$ is Keller domains in dimension $n$ for every $n\in\mathbb{Z}_{\geq 1}$. 
\end{dfn}

The first observation is that any field is Keller finite. This is the following proposition.

\begin{prop} \label{case0} Any field $K$ is a Keller domain.
\end{prop}
\begin{proof} Passing to the an algebraic closure of $K$ we can assume $K$ algebraically closed. Let $R =K[x_1,\ldots,x_n]/\idd{f_1,\ldots,f_n}$ and let $\mathfrak{m}$ be a maximal ideal of $R$. By the \cite[Corollary 11.15]{poly7} we know that $\dim R_{\mathfrak{m}}\leq \dim_{K} \mathfrak{m}/\mathfrak{m}^{2}$ and using the jacobian criterion we have $\dim_{K} \mathfrak{m}/\mathfrak{m}^{2} = n - \mbox{\textbf{rank}}(Jf(\mathfrak{m})) = n -  n = 0$. In particular, $\dim R_{\mathfrak{m}} = 0$. So, $R$ is an Artinian $K$-algebra. In particular, $\dim_{K} R$ is finite.
\end{proof}

\begin{prop}\label{nokeller} Consider the local ring $R = \mathbb{F}_q[[t]]$ of power series with coefficients on $\mathbb{F}_q$. Then $R$ is not a Keller-finite domain.
\end{prop}

\begin{proof} We will show that for every $n\in \mathbb{Z}_{>0}$ there are $f_1,\ldots,f_n\in R[x_1,\ldots,x_n]$ with jacobian $1$ such that the quotient 
$$
    S = R[x_1,\ldots,x_n]/\idd{f_1,\ldots,f_n}
$$
is not finitely generated as a $R$-module. Since $S$ is a $R$-algebra note that $S$ is finitely generated as $R$-module if and only if $S$ is integral over $R$. Pick the polynomials
$$
f_i = x_i-tx_i^{q} \quad \mbox{for every $i \in \{1,\ldots,n\}$}.
$$

Note that $\overline{x_i}\in S$ is not integral over $R$ for every $i$. Indeed, suppose that there is a relation:
$$
\overline{x_i}^{n}+a_{n-1}\overline{x_i}^{n-1}+\cdots+a_0 = 0
$$
in $S$ for some $i$. We can assume $i=1$. Then we have 
$$
x_{1}^{n}+a_{n-1}x_1^{n-1}+\cdots+a_{0} \in \idd{x_1-tx_1^{q},x_2-tx_2^{q},\ldots, x_n-tx_n^{q}}.
$$
So, there are $u_1(t;x_2,\ldots,x_n), \ldots, u_n(t;,x_2,\ldots,x_n) \in \mathbb{F}_p[[t]][x_2,\ldots,x_{n}]$ such that
$$
x_{1}^{n}+a_{n-1}x_1^{n-1}+\cdots+a_{0} = \sum_{i=0}^{m}u_{i}(t;x_2,\ldots,x_n)x_1^{i}-t\sum_{i=0}^{m}u_{i}(t;x_2,\ldots,x_n)x_1^{q+i}
$$
We conclude by looking at the right side of the equation above that the leading monomial is of the form $tu_{k}(t;x_2,\ldots,x_n)x_{1}^{q+k}$ for some $k$. But the left side is monical on $x_1$. This gives us a contradiction.
\end{proof}

In the next proposition, we explore local domains.

\begin{OBS} Let $(R,\mathfrak{m},k)$ be a local ring with residue field $k$ and fraction field $K$. Let $M$ be a $R$-module such that $M\otimes k$ and $M\otimes K$ are finite dimensional vector spaces. This is not enough to conclude that $M$ is a finitely generated $R$-module.
\end{OBS}

\begin{ex} Let $p$ be a prime number and consider the $p$-adic ring $\mathbb{Z}_p$. Consider the $\mathbb{Z}_p$-module
$$
M = \mathbb{Q}_p = \mathbb{Z}_p[1/p]  = \mathbb{Z}_p[t]/\idd{tp-1}.
$$
Note that $M$ is not finitely generated as $\mathbb{Z}_p$-module since $\overline{t}$ is not integral over $\mathbb{Z}_p$. But $\dim_{k} M\otimes k = 0$ and $\dim_{K} M\otimes K = 1$.
\end{ex}

\begin{prop} Let $(R,\mathfrak{m},k)$ be a discrete valuation ring with residue field $k$, fraction field $K$ and uniformizer $t \in R$. Let $f_1,\ldots,f_n \in R$ with the jacobian $1$. Let $S = R[x_1,\ldots,x_n]/\idd{f_1,\ldots,f_n}$ the quotient. Then, $\dim_{k} S\otimes k \leq \dim_{K} S\otimes K$
\end{prop}
\begin{proof} Let $\{u_1,\ldots, u_n\}$ be elements of $R$ such that $\{\overline{u}_1,\ldots,\overline{u}_n\}$ is $k$-independent. We will show that $\{u_1,\ldots, u_n\}$ is $K$-independent. Suppose that there are $a_1,\ldots,a_n \in K$ such that 
$$
    a_1u_1+\cdots+a_nu_n = 0.
$$
Cleaning denominators we can assume $a_i\in R$ for every $i$. Moreover, we can assume that $a_i \mod t \neq 0$ for some $i$. So, reducing $\mod \mathfrak{m}$ we get a non-trivial linear relation over $k$: $\overline{a_1}\cdot \overline{u_1}+\cdots+\overline{a_n}\cdot\overline{u_n} =0$, a contradiction.
\end{proof}

In the next theorem, we show that any complete discrete valuation ring of characteristic zero does satisfy the Keller-finite condition. We start with the following lemma.

\begin{lemma}\label{completion} Let $R$ be a ring and $M$ be a $R$-module. Let $I\subset R$ be an ideal and assume that
\begin{itemize}
    \item $R$ is complete with the $I$-adic topology,
    \item $\bigcap_{n\geq 1}I^{n}M = (0)$, and
    \item $M/IM$ is a $R/I$-module finitely generated
    \end{itemize}
Then $M$ is finitely generated as a $R$-module.
\end{lemma}

\begin{proof} See \cite[\href{https://stacks.math.columbia.edu/tag/00M9}{Lemma 10.96.12}]{stacks-project}.
\end{proof}

\begin{ex} \label{exemplinho} We have seen in Proposition \ref{nokeller} that the ring $R = \mathbb{F}_p[[t]]$ is not a Keller-finite domain: the $R$-module $M = R[x]/\idd{tx^{p}-x}$ is not a finitely generated. But, observe that 
\begin{itemize}
    \item $R$ is complete with respect to $\idd{t}$-topology
    \item $M/\idd{t}M = \mathbb{F}_p$ is finitely generated as $R/\idd{t} = \mathbb{F}_p$-module.
\end{itemize}
But, note that $\bigcap_{n\geq 1}I^{n}M \neq (0)$. Indeed, we have $x \in \bigcap_{n\geq 1}\idd{t}^{n}M$ since for every $n\in \mathbb{N}$ we have:
$$
x = tx^{p} = (tx^{p-1})x = tx^{p-1}tx^{p} = t^{2}x^{2p-1} = t^{2}x^{2p-2}(tx^{p}) = t^{3}x^{3p-2} = \cdots =t^{n}x^{np-(n-1)}
$$
So, $x \in \bigcap_{n\geq 1}\idd{t}^{n}M$ and the intersection is not zero.
\end{ex}

\begin{thm}\label{kellerfinite} Any complete discrete valuation ring of characteristic zero, $(R,\mathfrak{m},\field)$, is a Keller-finite domain.
\end{thm}

\begin{proof}  Let $f_1, \ldots, f_n \in R[x_1,\ldots,x_n] $ with invertible jacobian and consider the quotient $$A = R[x_1,\ldots,x_n]/\idd{f_1,\ldots,f_n}.$$ We must to show that $A$ is a finitely generated $R$-module. By the Lemma \ref{completion}  it is sufficient to show that the intersection of every power of the ideal $\mathfrak{m}A$ is trivial, where $t$ is a generator of $\mathfrak{m}$. Indeed,  by hypothesis $R$ is complete and 
$$
A/\mathfrak{m}A  = A\otimes R/\mathfrak{m} = A\otimes k = k[x_1,\ldots,x_n]/\idd{\overline{f}_1,\ldots,\overline{f}_n}
$$
is a $k$-vector space of finite dimension by Lemma \ref{case0}.

Now, denote by $J$ the intersection of all power of $\mathfrak{m}$ in $A$, that is $J = \bigcap_{n\geq 1}\idd{t}^{n}A$. By the Krull Intersection Theorem (see \cite[Theorem 3.16]{milne2020primer}, it is sufficient to show that $\mathfrak{m}$ is inside in every maximal ideal of $A$, that is $\mathfrak{m}\subset  \bigcap_{\mathfrak{p} \in \spm(A)}\mathfrak{p}$. But, since $R$ is a Jacobson ring so is $A$, because $A$ is finitely generated as $R$-algebra (see \cite[Theorem 4.19]{eisenbud2013commutative}). In particular, the "contraction" ( = inverse image under the quotient map) of the maximal ideal of $A$ to $R$ is also a maximal ideal. But, $R$ is local with maximal ideal $\idd{t}$. So, if $\mathcal{p}$ is a maximal ideal of $A$ then, $R\cap \mathfrak{p} = \idd{t}$. But, this means that $t$ is inside  $\mathfrak{p}$. So, we are done by the Lemma \ref{completion}.
\end{proof}

A consequence is the following theorem

\begin{thm}\label{main} The Jacobian Conjecture  is true if and only if $\mathbb{Z}_p$ is unimodular \textbf{for some prime $p$}.
\end{thm}

\begin{proof} If the Jacobian Conjecture is true then $\mathbb{Z}_p$ is an unimodular domain by Proposition \ref{UJC}. Suppose that $\mathbb{Z}_p$ is unimodular for some prime $p$. Assume, by contradiction, that the Jacobian Conjecture is false. By \cite[Proposition 1.1.19]{poly0} we know that there is a counterexample in the form $f = (f_1,\ldots,f_n) \in \mathcal{MP}_n(\overline{\mathbb{Q}}_p)$ with jacobian $1$ and coefficients on $\mathbb{Z}_p$. In particular, $f$ is not injective by Theorem \ref{CR}. So, there is $\alpha_1 \neq \alpha_2 \in L^n$ such that $f(\alpha_1) = f(\alpha_2)$ for some finite extension $L|\mathbb{Q}_p$. By a translation, we can assume $f(\alpha_1) = 0$. Now, by the Theorem \ref{kellerfinite} we have $\mathcal{O}_{L}$ a Keller-finite domain. So, $\mathcal{O}_{L}[x_1,\ldots,x_n]/\idd{f_1,\ldots,f_n}$ finitely generated as an $\mathcal{O}_{L}$-module. By the valuative criterion (see \cite[Theorem 4.7]{poly12}) there exists a bijection $X(\mathcal{O}_L) \cong X(L)$, where $X$ is the affine scheme defined by the equations $f_1,\ldots, f_n$. But, since we are assuming that $\mathbb{Z}_p$ is unimodular then by Theorem \ref{fiber2} it follows that $f\otimes\mathcal{O}_{L}$ is an injective polynomial map. In particular, $\#X(L) = \#X(\mathcal{O}_L) \leq 1$. Contradiction, since $\#X(L)\geq 2$ ($\alpha_1,\alpha_2 \in X(L)$).
\end{proof}

\noindent \textbf{Acknowledgements.} W.Mendson  thanks Ehsan Tavanfar for the discussion and references of useful topics in the proof of the Theorem \ref{kellerfinite}. The author acknowledges the support of CNPq and Universidade Federal Fluminense (UFF).

\bibliographystyle{amsplain}
\bibliography{annot}

\end{document}